\documentclass[10pt]{article}
\usepackage{amsmath, amsthm, amssymb, amsfonts, mathrsfs, mathtools}
\usepackage{graphicx}
\usepackage{makeidx}
\usepackage[left=2cm,right=2cm,top=2cm,bottom=2cm]{geometry}
\usepackage{caption}
\usepackage{subcaption}
\usepackage{tikz}
\usetikzlibrary{decorations.pathreplacing}
\tikzstyle{vertex}=[auto=left,circle,draw=black,fill=white, inner sep=1.5]

\usepackage{hyperref}
\hypersetup{colorlinks=true, citecolor={blue}, linkcolor=blue, filecolor=magenta, urlcolor=cyan}

\newtheorem{theorem}{Theorem}[section]

\newtheorem{lemm}{Lemma}[section]
\newtheorem{definition}{Definition}

\title{Maximum Frustration in Signed Generalized Petersen Graphs}

\author{ Deepak Sehrawat\\
Department of Mathematics\\
Indian Institute of Technology Guwahati\\
Guwahati, India - 781039\\
Email: deepakmath55555@iitg.ac.in\\
\\
Bikash Bhattacharjya\\
Department of Mathematics\\
Indian Institute of Technology Guwahati\\
Guwahati, India - 781039\\
Email: b.bikash@iitg.ac.in
}
\date{}

\begin{document}
\maketitle

\vspace{-0.3in}

\vspace*{0.3cm}
\noindent
\textbf{Abstract.} A \textit{signed graph} is a simple graph whose edges are labelled with positive or negative signs. A cycle is \textit{positive} if the product of its edge signs is positive. A signed graph is \textit{balanced} if every cycle in the graph is positive. The \textit{frustration index} of a signed graph is the minimum number of edges whose deletion makes the graph balanced. The \textit{maximum frustration} of a graph is the maximum frustration index over all sign labellings. In this paper, first, we prove that the maximum frustration of generalized Petersen graphs $P_{n,k}$ is bounded above by $\left\lfloor \frac{n}{2} \right\rfloor + 1$ for $\gcd(n,k)=1$, and this bound is achieved for $k=1,2,3$. Second, we prove that the maximum frustration of $P_{n,k}$ is bounded above by $d\left\lfloor \frac{n}{2d} \right\rfloor + d + 1$, where $\gcd(n,k)=d\geq2$.

\vspace*{0.2cm}

\section{Introduction}

In~\cite{Harary2}, Harary introduced the notion of signed graphs and balance. A cycle in a signed graph is called \textit{positive} if the product of the signs of its edges is positive. A signed graph is \textit{balanced} if all its cycles are positive. Harary also gave a necessary and sufficient condition for a signed graph to be balanced. Two years after Harary's paper, Catwright and Harary \cite{Harary} used signed graphs and balance to model social stress in groups of people in social psychology.

A signed graph is \textit{unbalanced} if it is not balanced, \textit{i.e.,} it has at least one cycle which is not positive. An unbalanced signed graph can be made balanced by deleting edges. The smallest number of edges whose deletion leaves the graph balanced is called the \textit{frustration index}. This number is implicated in certain questions of social psychology~\cite{Abelson} and spin-glass physics~\cite{Toulouse}.

Finding frustration index is an NP-complete problem~\cite{Barahona2}. In~\cite{T.Zaslavsky}, the author determined the frustration indices of all signed Petersen graphs. In~\cite{Sivaraman}, the author determined the frustration indices of all signed Heawood graphs. And in~\cite{Bowlin}, Bowlin gave the upper bound for the maximum frustration of complete bipartite graph. Many other classes of graphs are treated in \cite{Sole}.

In \cite{Sivaraman}, Sivaraman proved that the frustration index of any signed graph $ \Sigma$, whose underlying graph $G$ is cubic, simple and triangle-free, is bounded above by $\frac{3}{8}|V(G)|$, where $|V(G)|$ denotes the number of vertices of $G$. In this paper, we find an upper bound for the maximum frustration of generalized Petersen graphs.

This paper is organised as follows. In the next section we describe some basic definitions and existing results about the frustration index. In Section~\ref{Bound_gcd(n,k) = 1}, we prove that the maximum frustration of generalised Petersen graphs $P_{n,k}$, where $\gcd(n,k)= 1$, is bounded above by $\lfloor \frac{n}{2} \rfloor + 1$, and that this bound is tight when $k=1,2, ~\text{and}~ 3$. In Section~\ref{Bound_gcd(n,k) = d}, we show that the maximum frustration of $P_{n,k}$, where $\gcd(n,k)= d \geq2$, is bounded above by $d\left\lfloor \frac{n}{2d} \right\rfloor + d + 1$. Finally, we list some problems in the last section for future work.

Throughout this paper we consider only simple and finite graphs. For the graph theoretic terms that are used but not defined in this paper, see~\cite{Bondy}. For detailed study of signed graphs, we refer the reader to \cite{Zaslavsky}. In all the figures of this paper, solid lines represent positive edges and dotted lines represent negative edges.

\section{Preliminaries}

Let $G = (V(G), E(G))$ be a simple graph. Further, by $|V(G)|$ and $|E(G)|$ we denote the cardinality of the vertex set and the edge set of $G$, respectively. A \textit{signed graph} is a graph whose edges are labelled with positive or negative signs. We write $\Sigma = (G, E_{-})$ to denote a signed graph, where $G$ is called the \textit{underlying graph} of $\Sigma$ and $E_{-}$ denotes the set of negative edges. The set $E_{-}$ is called the \textit{signature}. The number of edges in a signature $E_{-}$ is called the size of the signature and denoted by $|E_{-}|$. If the edges of $ \Sigma$ are all positive, that is $E_{-}=\emptyset$, then the signed graph is called the \textit{all positive} signed graph.

A \textit{switching} of a vertex of a signed graph is to change the sign of each edge incident to the vertex. If we switch each vertex of a subset $X$ of $V(G)$, then we write the resulting signed graph as $\Sigma^{X}$. We say $\Sigma_{1}$ is \textit{switching equivalent} or simply \textit{equivalent} to $\Sigma_{2}$, denoted $\Sigma_{1} \sim \Sigma_{2}$, if both $\Sigma_{1} ~ \text{and} ~ \Sigma_{2}$ have same underlying graph $G$ and $\Sigma_{1} = \Sigma_{2}^{X}$ for some $X \subseteq V(G)$. Let $E_{-}^{1} ~ \text{and} ~ E_{-}^{2}$ be the signatures of $\Sigma_{1} ~ \text{and} ~ \Sigma_{2}$, respectively, then we also say that $E_{-}^{1}$ is switching equivalent to $E_{-}^{2}$, denoted $E_{-}^{1} \sim E_{-}^{2}$, to mean that $\Sigma_{1}$ is switching equivalent to $\Sigma_{2}$. Note that switching defines an equivalence relation on the set of all signed graphs over $G$ (also on the set of signatures). Each equivalence class of this equivalence relation is denoted by $[G,E_{-}]$, where $(G,E_{-})$ is any member of the class. In short, we write $[E_{-}]$ to denote the class of all signatures equivalent to the signature $E_{-}$. 

The \textit{sign} of a cycle in a signed graph is the product of its edge signs. A signed graph is \textit{balanced} if each of its cycles is positive. The following theorem states that the set of negative cycles uniquely determines the equivalence class to which a signed graph belongs.

\begin{theorem}~\cite{Zaslavsky}
\label{Signature}
Two signed graphs $\Sigma_{1}$ and $\Sigma_{2}$ are switching equivalent if and only if they have the same set of negative cycles.
\end{theorem}

Now we define the frustration index and the frustration number of a signed graph. These parameters are invariants under switching because they depend only on cycle signs, which are not changed by switching.

\begin{definition}\label{FI}
\rm{The \textit{frustration index} of a signed graph $(G, E_{-})$, denoted $l(G, E_{-})$, is the smallest number of edges whose deletion leaves a balanced signed graph.}
\end{definition}
 
Implicit in \cite{Barahona}, frustration index is switching invariant. But in the following lemma, we give a simple and different proof of this fact.

\begin{lemm}\label{invariant FI}
The frustration index of a signed graph is invariant under switching.
\end{lemm}
\begin{proof}
Let $(G, E_{-})$ be a signed graph. Suppose $(G, E_{-}) - D$ is a balanced signed graph, where $D \subseteq E(G)$. Since switching does not change the sign of a cycle, if we switch to $(G, E_{-})^{X}$ and then delete $D$, there will still not be any negative cycles. Therefore $(G, E_{-})^{X} - D$ is balanced. Thus deletion of an edge set that makes one of the graphs in $[G, E_{-}]$ balanced, also make the other graphs of $[G, E_{-}]$ balanced. This completes the proof.
\end{proof}

\noindent
If $(G, E_{-})$ is balanced then clearly $l(G, E_{-}) = 0$. Also it is easy to see that 
\begin{equation}\label{min ineq}
l(G, E_{-}) = \underset{E_{-}' \in [E_{-}]}{\min} |E_{-}'|.
\end{equation}

The \textit{maximum frustration} $D(G)$ of a graph $G$ is the maximum frustration index over all possible signatures. That is, \begin{equation}\label{max ineq}
D(G) = \underset{E_{-} \subseteq E(G)}{\max} l(G,E_{-}).
\end{equation}

\begin{definition}\label{FN}
\rm{The \textit{frustration number} of a signed graph $(G, E_{-})$, denoted $l_{0}(G, E_{-})$, is the smallest number of vertices whose deletion leaves a balanced signed graph.}
\end{definition}

\begin{lemm}\cite{T.Zaslavsky}\label{invariant FN}
Switching does not change the frustration number of a signed graph. Moreover, $l_{0}(G, E_{-}) \leq l(G, E_{-})$ for every signature $E_-$ of a graph $G$.
\end{lemm}

From here onwards, in view of Lemma~\ref{invariant FI} and Lemma~\ref{invariant FN}, we denote the signed graph by $[G, E_{-}]$ as we study only the frustration index and the frustration number of signed graphs in this paper. In~\cite{Sivaraman}, the author proved that the frustration index and the frustration number are same for signed cubic graphs.

\begin{theorem}\cite{Sivaraman} \label{FI = FN}
Let $[G, E_{-}]$ be a signed cubic graph. Then $l[G, E_{-}]  = l_{0}[G, E_{-}]$. 
\end{theorem}

Let $[G, E_{-}]$ be a signed graph. A signature $E_{-}' \in [E_{-}]$ is said to be \textit{minimum} if the number of edges in $E_{-}'$ is minimum among all signatures in $[E_{-}]$. Note that a minimum signature need not be unique. This fact can be verified by looking at a negative cycle. A signed graph with a minimum signature is said to be a \textit{reduced} signed graph. Further, the edges of a minimum signature of cubic graph has a special structure. More precisely, we have the following lemma.

\begin{lemm}\cite{T.Zaslavsky} \label{minimal}
The edges of any minimum signature of a cubic graph form a matching. 
\end{lemm}

It is important to note that the frustration index is equal to the size of a minimum signature for every signed graph. The following result, that easily follows from Lemma~\ref{minimal}, says that the frustration index of a signed cubic graph $[G, E_{-}]$ is at most $\frac{|V(G)|}{2}$. 

\begin{theorem}\label{bound0}
Let $[G, E_{-}]$ be a signed cubic graph. Then $l[G, E_{-}] \leq \frac{|V(G)|}{2}$.
\end{theorem}

The bound obtained in Theorem~\ref{bound0} has an improvement for the cubic, triangle-free graphs.

\begin{theorem}\cite{Sivaraman}\label{bound1}
Let $[G, E_{-}]$ be a signed graph whose underlying graph $G$ is simple, cubic and triangle-free. Then $l[G, E_{-}] \leq \frac{3}{8}|V(G)|$.
\end{theorem}

\begin{definition}\label{GPG}
\rm{Let $n$ and $k$ be positive integers such that  $2 \leq 2k <n$. Then the \textit{generalised Petersen graph} $P_{n,k}$ is defined to have the vertex set $V(P_{n,k}) = \{u_{i}, v_{i}~|~i\in[n]\},$ and edge set \linebreak[4] $ E(P_{n,k}) = \{u_{i}u_{i+1}, v_{i}v_{i+k},u_{i}v_{i}~|~i\in[n]\},$ where $[n]=\{0,1,...,n-1\}$ and the subscripts are read modulo $n$.}
\end{definition}

For instance a signed generalised Petersen graph $P_{7,1}$ with a minimum signature of size four is shown in Figure~\ref{P7}. The fact that the signature of Figure~\ref{P7} is minimum is proved in Lemma~\ref{tight1}.


%


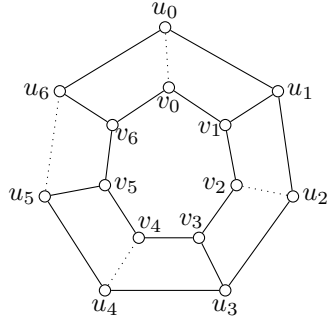
\begin{figure}[ht]
\centering
\begin{tikzpicture}[scale=0.5]
\node[vertex] (v1) at (11.9,7) {};\node[vertex] (v2) at (14.9,5.3) {};\node[vertex] (v3) at (15.3,2.5) {};\node[vertex] (v4) at (13.5,0) {};\node[vertex] (v5) at (10.3,0) {};
\node[vertex] (v6) at (8.7,2.5) {};\node[vertex] (v7) at (9.1,5.3) {};\node[vertex] (v8) at (12.0,5.4) {};\node[vertex] (v9) at (13.5,4.4) {};\node[vertex] (v10) at (13.8,2.8){};
\node[vertex] (v11) at (12.8,1.4) {};\node[vertex] (v12) at (11.2,1.4) {};\node[vertex] (v13) at (10.3,2.8) {};\node[vertex] (v14) at (10.5,4.4) {};

\node [above] at (v1) {$u_{0}$};\node [right] at (v2) {$u_{1}$};\node [right] at (v3) {$u_{2}$};\node [below] at (v4) {$u_{3}$};\node [below] at (v5) {$u_{4}$};
\node [left] at (v6) {$u_{5}$};\node [left] at (v7) {$u_{6}$};\node [below] at (v8) {$v_{0}$};\node at (13.10,4.3) {$v_{1}$};\node [left] at (v10) {$v_{2}$};
\node at (12.6,1.85) {$v_{3}$};\node at (11.5,1.75) {$v_{4}$};\node [right] at (v13) {$v_{5}$};\node at (10.9,4.1) {$v_{6}$};

\foreach \from/\to in {v1/v2,v2/v3,v3/v4,v4/v5,v5/v6,v7/v1,v8/v9,v9/v10,v10/v11,v11/v12,v12/v13,v13/v14,v14/v8,v2/v9,v4/v11,v6/v13,v7/v14} \draw (\from) -- (\to);

\draw [dotted] (v6) -- (v7);\draw [dotted] (v1) -- (v8);\draw [dotted] (v3) -- (v10);\draw [dotted] (v5) -- (v12);

\end{tikzpicture}
\caption{A signed $P_{7,1}$ with a signature of size four}\label{P7}
\end{figure}

We call the vertices $u_{0},u_{1},...,u_{n-1}$ to be \textit{u-vertices} and the vertices $v_{0},v_{1},...,v_{n-1}$ to be \textit{v-vertices}. From the definition, it is clear that $P_{n,k}$ is a cubic graph and $P_{5,2}$ is the well-known Petersen graph. The edges $u_{i}v_{i}$ for $i \in [n]$ are called the \textit{spokes} and the set of spokes is denoted by $S_{s}$. The cycle induced by $u\text{-vertices}$ is called the \textit{outer cycle} of $P_{n,k}$ and is denoted by $C_{o}$. The cycle(s) induced by $v\text{-vertices}$ is(are) called the \textit{inner cycle(s)} of $P_{n,k}$. If $\text{gcd}(n,k) = d$ then the subgraph induced by $v\text{-vertices}$ consists of $d$ pairwise disjoint $\frac{n}{d}$-cycles. If $d >1$ then no two vertices among $v_{0}, v_{1},..., v_{d-1}$ can be in the same $\frac{n}{d}$-cycle.

\section{Upper Bound of $D(P_{n,k})$ for $\gcd(n,k) = 1$}\label{Bound_gcd(n,k) = 1}

\begin{definition}\label{$G_n$}
\rm{For $n\geq3$, the graph $G_{n}$ is defined to have the vertex set $V(G_{n}) = \{u_{i},v_{i}~|~i\in[n] \}$, and edge set $ E(G_{n}) = \{u_{i}u_{i+1}, u_{i}v_{i}~|~i\in[n]\},$ where the subscripts are read modulo $n$.}
\end{definition}

\begin{figure}[h]
\centering
\begin{tikzpicture}
\node[vertex] (v1) at (0,0) {};
\node [left] at (v1) {$u_{0}$};
\node[vertex] (v2) at (1,0) {};
\node [above] at (v2) {$u_{1}$};
\node[vertex] (v3) at (2,0) {};
\node [above] at (v3) {$u_{2}$};
\node[vertex] (v4) at (3,0) {};
\node [above] at (v4) {$u_{3}$};
\node[vertex] (v5) at (4,0) {};
\node [right] at (v5) {$u_{4}$};
\node[vertex] (v6) at (0,-1) {};
\node [below] at (v6) {$v_{0}$};
\node[vertex] (v7) at (1,-1) {};
\node [below] at (v7) {$v_{1}$};
\node[vertex] (v8) at (2,-1) {};
\node [below] at (v8) {$v_{2}$};
\node[vertex] (v9) at (3,-1) {};
\node [below] at (v9) {$v_{3}$};
\node[vertex] (v10) at (4,-1) {};
\node [below] at (v10) {$v_{4}$};
\draw (3.97,0.06) to[out=140,in=40] (0.03,0.06);
\foreach \from/\to in {v1/v2,v2/v3,v3/v4,v4/v5,v1/v6,v2/v7,v3/v8,v4/v9,v5/v10} \draw (\from) -- (\to);
\end{tikzpicture}
\caption{The graph $G_{5}$}
\label{Figure2}
\end{figure}
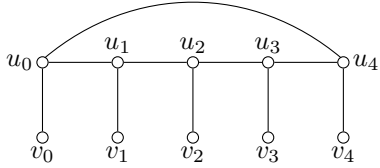

 For instance, the graph $G_{5}$ is shown in Figure~\ref{Figure2}. From the definition, it is clear that the graph $G_{n}$ is a sub-cubic graph, \textit{i.e.,} each vertex of $G_{n}$ has degree at most three. Note that $G_{n}$ is isomorphic to a subgraph of $P_{n,k}$ and that subgraph can be obtained by deleting all the edges of the inner cycle(s) of $P_{n,k}$. Let $C_{n}$ denote the cycle $u_{0}u_{1}...u_{n-1}u_{0}$ of $G_{n}$.

By Lemma~\ref{minimal}, it is clear that the edges of any minimum signature of $G_{n}$ form a matching. Therefore a minimum signature of $G_{n}$ is either an empty signature or a signature of size one, since the cycle $C_{n}$ can be either positive or negative, and any negative edge incident to a vertex $v_{i}$ can be made positive by switching that vertex. This implies that $l[G_{n}, E_{-}] \leq 1$.

We say $[G_{n}, E_{-}]$ is balanced or unbalanced according as the cycle $C_{n}$ is positive or negative, respectively. Further, if we do not allow switching operation by $v\text{-vertices}$ then the following lemma gives an upper bound on the size of a minimum signature of $G_{n}$.

\begin{lemm}\label{lemma1o4}
Let $[G_{n}, E_{-}]$ be a signed graph and let switching operation be not allowed by $v\text{-vertices}$. Then the maximum size of a minimum signature is either $\left\lfloor \frac{n}{2} \right\rfloor$ or $\left\lfloor \frac{n}{2} \right\rfloor +1$ according as $C_{n}$ is positive or negative in $[G_{n}, E_{-}]$, respectively.
\end{lemm}
\begin{proof}
Let $[G_{n}, E_{-}]$ be a signed graph. Also let $E_{-}'$ be a minimum signature, switching equivalent to $E_{-}$, such that edges of $E_{-}'$ form a matching in $G_{n}$. The existence of such $E_{-}'$ is assured by Lemma~\ref{minimal}.

If the cycle $C_{n}$ is positive in $[G_{n}, E_{-}]$ then by \textit{switching} $u\text{-vertices}$, if needed, we can make all the edges of $C_{n}$ positive. Now the negative edges in the resulting signed graph are of the form $u_{i}v_{i}$ only. Therefore it is obvious that the maximum number of edges in a minimum signature is $\left\lfloor \frac{n}{2} \right\rfloor$. Hence $|E_{-}'| \leq \left\lfloor \frac{n}{2} \right\rfloor$.

If the cycle $C_{n}$ is negative in $[G_{n}, E_{-}]$ then by \textit{switching operation}, if needed, we can make any pre-chosen edge of $C_{n}$ negative and rest of the edges positive. Let that negative edge of $C_{n}$ be $u_{0}u_{1}$. Let $E_{-}'$ be a minimum signature of the resulting signed graph. Note that the edges of $E_{-}'$ form a matching in which one edge is $u_{0}u_{1}$ and other edges of $E_{-}'$ are some $u_{i}v_{i}\text{'s}$. Now take $E_{-}'' = E_{-}' \setminus \{u_{0}u_{1}\}$, then it is obvious that $|E_{-}'| = |E_{-}''|+1$. Further note that $E_{-}'' \subseteq S$, where $S= \{u_{2}v_{2}, u_{3}v_{3},...,u_{n-1}v_{n-1}\}$ and $|S|= n-2$. To complete the proof, we need to show that $|E_{-}''| \leq \left\lfloor \frac{n}{2} \right\rfloor$.

Suppose on the contrary that $|E_{-}''| \geq \left\lfloor \frac{n}{2} \right\rfloor +1$. Therefore we have 
\begin{equation}\label{equation1}
|E_{-}'| \geq \left\lfloor \frac{n}{2} \right\rfloor +2.
\end{equation}

Since $u_{0}u_{1} \in E_{-}'$ and $E_{-}'' \subseteq S$ so by switching each vertex of the set $\{u_{1}, u_{2},...,u_{n-1}\}$, we get a signature $E_{-}^{*}$ such that $E_{-}^{*} \sim E_{-}'$ and $E_{-}^{*} = (S \setminus E_{-}'') \cup \{u_{1}v_{1}, u_{n-1}u_{0}\}$. This implies that
\begin{equation*}
\begin{split}
|E_{-}^{*}| & = |S| - |E_{-}''| +2 \\
 & = n - |E_{-}''| \\
 & \leq n - \Big(\left\lfloor \frac{n}{2} \right\rfloor +1\Big) \\ & \leq \left\lfloor \frac{n}{2} \right\rfloor.
\end{split}
\end{equation*}

Thus we have a signature $E_{-}^{*}$ such that $E_{-}^{*} \sim E_{-}'$ and $|E_{-}^{*}| \leq \left\lfloor \frac{n}{2} \right\rfloor$. This is a contradiction to the fact that $E_{-}'$ is minimum and $|E_{-}'| \geq \left\lfloor \frac{n}{2} \right\rfloor +2$. Thus $|E_{-}''| \leq \left\lfloor \frac{n}{2} \right\rfloor$, and this completes the proof.
\end{proof}

Now we find an upper bound for the maximum frustration of $P_{n,k}$ in terms of $n$, where $\gcd(n,k) = 1$. 

\begin{theorem}\label{gcd(n,k)=1}
Let $\gcd(n,k) = 1$. Then $D(P_{n,k}) \leq \left\lfloor \frac{n}{2} \right\rfloor +1.$
\end{theorem}
\begin{proof}
Let $[P_{n,k}, E_{-}]$ be a signed generalised Petersen graph, where $\gcd(n,k) = 1$. Note that $P_{n,k}$ has only one inner cycle of length $n$ since $\gcd(n,k) = 1$. By $C_{o}$ and $C_{i}$ we denote the outer cycle and the inner cycle of $P_{n,k}$, \textit{i.e.,} $C_{o} = u_{0}u_{1}...u_{n-1}u_{0}$ and $C_{i} = v_{0}v_{k}v_{2k}...v_{n-k}v_{0}$, respectively. 

Let $G_{o}$ be the subgraph of $P_{n,k}$ such that $V(G_{o}) = V(P_{n,k})$ and $E(G_{o}) = E(C_{o}) \cup S_{s}$. Also let $G_{i}$ be the subgraph of $P_{n,k}$ such that $V(G_{i}) = V(P_{n,k})$ and $E(G_{i}) = E(C_{i}) \cup S_{s}$. Note that the graphs $G_{o}$ and $G_{i}$ are both isomorphic to $G_{n}$. We prove the theorem by considering the following cases.

\vspace*{0.3cm}
\noindent
\textbf{Case 1.} Assume that both the cycles $C_{o}$ and $C_{i}$ are positive in $[P_{n,k}, E_{-}]$. By switching $u\text{-vertices}$ and $v\text{-vertices}$, if needed, we can make all the edges of both $C_{o}$ and $C_{i}$ positive. Thus all the edges of the resulting signature are spokes only. Therefore we have $l[P_{n,k}, E_{-}] < \left\lfloor \frac{n}{2} \right\rfloor +1$, since out of $n$ spokes at most $\left\lfloor \frac{n}{2} \right\rfloor$ spokes can be negative (up to switching).

\vspace*{0.3cm}
\noindent
\textbf{Case 2.} Assume that the cycle $C_{o}$ is positive and $C_{i}$ is negative in $[P_{n,k}, E_{-}]$. Without loss of generality, assume that all the edges of $C_{o}$ are positive. Next, if some vertex $v_{i}$ is incident to more than one negative edge then we switch that vertex $v_{i}$ to reduce the number of negative edges. We keep doing this operation until every vertex $v_{i}$ is incident to at most one negative edge. Now let $[P_{n,k}, E_{-}']$ denote the resulting signed graph, where $E_{-}'$ is switching equivalent to $E_{-}$. The edges of $E_{-}'$ form a matching and $E_{-}' \subset E(G_{i})$. Note that to reduce the signature $E_{-}'$ we do not need to switch any vertex $u_{i}$. Since $G_{i}$ is isomorphic to $G_{n}$ and we do not switch any $u\text{-vertex}$, so by Lemma~\ref{lemma1o4}, the maximum number of negative edges in a minimum signature, equivalent to $E_{-}'$, is $\left\lfloor \frac{n}{2} \right\rfloor +1$. This implies that $l[P_{n,k}, E_{-}] \leq \left\lfloor \frac{n}{2} \right\rfloor +1$.

Similarly, if the cycle $C_{i}$ is positive and $C_{o}$ is negative in $[P_{n,k},  E_{-}]$ then also $l[P_{n,k}, E_{-}] \leq \left\lfloor \frac{n}{2} \right\rfloor +1$.

\vspace*{0.3cm}
\noindent
\textbf{Case 3.} Assume that both the cycles $C_{o}$ and $C_{i}$ are negative in $[P_{n,k}, E_{-}]$. By switching some $u\text{-vertices}$, if needed, we make the edge $u_{0}u_{1}$ negative and rest of the edges of $C_{o}$ positive. Next, we switch some $v_{i}\text{'s}$, if needed, such that only one edge of $C_{i}$ is negative and remaining edges are positive. Let that negative edge of $C_{i}$ be $v_{r}v_{r+k}$, for some $0 \leq r \leq n-1$.

We will complete the proof by showing that the number of negative edges in a minimum signature containing the edges $u_{0}u_{1}$ and $v_{r}v_{r+k}$ can be at most $\left\lfloor \frac{n}{2} \right\rfloor + 1$. 

Further, let $E_{-}'$ be a minimum signature, equivalent to $E_{-}$, such that edges of $E_{-}'$ form a matching and except the edges $u_{0}u_{1}$ and $v_{r}v_{r+k}$, its all other edges are some spokes only. Take $E_{-}'' = E_{-}' \setminus \{u_{0}u_{1}, v_{r}v_{r+k}\}$ so that $|E_{-}'| = |E_{-}''| +2$. Note that the edges of $E_{-}''$ are some spokes only. To complete the proof, it is enough to show that the $|E_{-}''| \leq \left\lfloor \frac{n}{2} \right\rfloor - 1$. To do so we consider two sub-cases.

\vspace*{0.3cm}
\noindent
\textbf{(a)} Let neither $v_{r}$ nor $v_{r+k}$ be equal to $v_{0}$ and $v_{1}$. Note that $E_{-}'' \subseteq S$, where $S$ is a subset of $S_{s}$. More precisely, $S = S_{s} \setminus \{u_{0}v_{0}, u_{1}v_{1}, u_{r}v_{r}, u_{r+k}v_{r+k}\}$ so that $|S| = n-4$. Suppose on the contrary that $|E_{-}''| \geq \left\lfloor \frac{n}{2} \right\rfloor$, and so $|E_{-}'| \geq \left\lfloor \frac{n}{2} \right\rfloor + 2$.

By switching each vertex of the set $\{u_{1}, u_{2},..., u_{n-1}, v_{r}\}$, we get an equivalent signature $E_{-}^{*}$, where $E_{-}^{*} = (S \setminus E_{-}'') \cup \{u_{1}v_{1}, v_{r-k}v_{r}, u_{r+k}v_{r+k}, u_{n-1}u_{0}\}$. Therefore we have 
\begin{equation*}
\begin{split}
|E_{-}^{*}| & = |S| - |E_{-}''| +4 \\
 & = n-4 - |E_{-}''| +4 \\
 & \leq n - \left\lfloor \frac{n}{2} \right\rfloor \\
 & \leq \left\lfloor \frac{n}{2} \right\rfloor +1.
\end{split}
\end{equation*} 
But this is a contradiction because $E_{-}'$ is minimum and $|E_{-}'| \geq \lfloor \frac{n}{2} \rfloor + 2$. Therefore $|E_{-}''| \leq \left\lfloor \frac{n}{2} \right\rfloor - 1$, and this implies that $|E_{-}'| \leq \left\lfloor \frac{n}{2} \right\rfloor + 1$. 

\vspace*{0.3cm}
\noindent
\textbf{(b)} Let one of the end points of the edge $v_{r}v_{r+k}$ be either $v_{0}$ or $v_{1}$. Let $v_{r} = v_{0}$, then the spoke $u_{n-1}v_{n-1}$ cannot be included in $E_{-}''$ because $E_{-}'$ is a minimum signature. Otherwise, if $u_{n-1}v_{n-1} \in E_{-}''$ then switching each vertex of the set $\{u_{n-1}, u_{0}, v_{0}\}$ will give us an equivalent signature of size one less than the size of $E_{-}'$, and this is a contradiction as $E_{-}'$ is minimum. 

\vspace*{0.3cm}
\noindent
\textbf{b(i)} Let $k>1$, then $E_{-}'' \subseteq S$, where $S = S_{s} \setminus \{u_{0}v_{0}, u_{1}v_{1}, u_{k}v_{k}, u_{n-1}v_{n-1}\}$ and $|S| = n-4$. Let us suppose on the contrary that $|E_{-}''| \geq \left\lfloor \frac{n}{2} \right\rfloor$ so that $|E_{-}'| \geq \left\lfloor \frac{n}{2} \right\rfloor + 2$.

By switching each vertex of the set $\{u_{1}, u_{2},..., u_{n-2}, v_{k}\}$, we get an equivalent signature $E_{-}^{*}$, where $E_{-}^{*} = (S \setminus E_{-}'') \cup \{u_{1}v_{1}, v_{k}v_{2k}, u_{n-2}u_{n-1}\}$. Therefore we have 
\begin{equation*}
\begin{split}
|E_{-}^{*}| & = |S| - |E_{-}''| +3 \\
 & = n-4 - |E_{-}''| +3 \\
 & \leq n-1 - \left\lfloor \frac{n}{2} \right\rfloor \\
 & \leq \left\lfloor \frac{n}{2} \right\rfloor.
\end{split}
\end{equation*} 

\vspace*{0.3cm}
\noindent
\textbf{b(ii)} If $k=1$ then $E_{-}'' \subseteq S$, where $S = S_{s} \setminus \{u_{0}v_{0}, u_{1}v_{1}, u_{n-1}v_{n-1}\}$ and $|S| = n-3$. Let us suppose on the contrary that $|E_{-}''| \geq \left\lfloor \frac{n}{2} \right\rfloor$ so that $|E_{-}'| \geq \left\lfloor \frac{n}{2} \right\rfloor + 2$.

By switching each vertex of the set $\{u_{1},u_{2},...,u_{n-2}, v_{1}\}$, we get an equivalent signature $E_{-}^{*}$ such that $E_{-}^{*} = (S \setminus E_{-}'') \cup \{v_{1}v_{2}, u_{n-2}u_{n-1}\}$. Therefore we have 
\begin{equation*}
\begin{split}
|E_{-}^{*}| & = |S| - |E_{-}''| +2 \\
 & = n-3 - |E_{-}''| +2 \\
 & \leq n-1 - \left\lfloor \frac{n}{2} \right\rfloor \\
 & \leq \left\lfloor \frac{n}{2} \right\rfloor.
\end{split}
\end{equation*}

\noindent
In both \textbf{b(i)} and \textbf{b(ii)}, we get a contradiction because $E_{-}'$ is minimum and $|E_{-}'| \geq \left\lfloor \frac{n}{2} \right\rfloor + 2$. Hence $|E_{-}''| \leq \left\lfloor \frac{n}{2} \right\rfloor - 1$ and this implies that $|E_{-}'| \leq \left\lfloor \frac{n}{2} \right\rfloor + 1$.

Similarly, if $v_{r} = v_{1}$ or $v_{r} = v_{n-k}$ or $v_{r} = v_{n-k+1}$ then also it can be shown that $|E_{-}'| \leq \left\lfloor \frac{n}{2} \right\rfloor + 1$. Hence from all these cases, we conclude that the  number of negative edges in a minimum signature, equivalent to $E_{-}$, is at most $\left\lfloor \frac{n}{2} \right\rfloor + 1$. Since  $E_{-}$ is an arbitrary signature, we have $$D(P_{n,k}) \leq \left\lfloor \frac{n}{2} \right\rfloor +1.$$ This completes the proof.
\end{proof}

\noindent
\textbf{Remark.} It is interesting to note that $P_{n,k}$ has a special structure when $k=1$. The graph $P_{n,1}$ is the $n\text{-gonal}$ prism. Therefore $P_{n,1}$ can be drawn as a ring of $n$ quadrangles. For example $P_{7,1}$ is the heptagonal prism and it has been drawn as a ring of seven quadrangles in Figure~\ref{P7}.

In the following lemma we show that there exist a signed $P_{n,1}$ whose frustration index is $\left\lfloor \frac{n}{2} \right\rfloor +1$.

\begin{lemm}\label{tight1}
There exists a signature $E_{-}$ such that $l[P_{n,1}, E_{-}] = \left\lfloor \frac{n}{2} \right\rfloor+1$.
\end{lemm}
\begin{proof}
Let $n=2k+1$ and consider a signed generalised Petersen graph $[P_{2k+1,1}, E_{-}^{1}]$, where \linebreak[4] $E_{-}^{1} = \{u_{0}v_{0}, u_{2}v_{2}, u_{4}v_{4},..., u_{2k-2}v_{2k-2}, u_{2k-1}u_{2k}\}$. Note that $|E_{-}^{1}| = k+1$. It is clear that the graph $P_{2k+1,1}$ has $2k+1$ quadrangles. With one negative edge, it can be made at most two quadrangles negative and therefore with $k$ or less than $k$ negative edges it can be made at most $2k$ quadrangles negative. But the edges of $E_{-}^{1}$ makes all quadrangles negative. Therefore signature $E_{-}^{1}$ is a minimum signature. Hence $l[P_{2k+1,1}, E_{-}^{1}] = k+1$. For example, see Figure~\ref{P7}.

Let $n=2k$ and consider a signed generalised Petersen graph $[P_{2k,1}, E_{-}^{2}]$, where \linebreak[4]  $E_{-}^{2} = \{u_{0}v_{0}, u_{2}v_{2}, u_{4}v_{4},..., u_{2k-4}v_{2k-4}, u_{2k-3}u_{2k-2}, v_{2k-2}v_{2k-1}\}$. Note that $|E_{-}^{2}| = k+1$, and we want to prove that $E_{-}^{2}$ is minimum. The graph $P_{2k,1}$ has $2k$ quadrangles. Further, $E_{-}^{2}$ makes the outer cycle, the inner cycle and all the quadrangles of $P_{2k,1}$ negative. Let $E_{-}^{3}$ be a signature equivalent to $E_{-}^{2}$. Assume that $E_{-}^{3}$ contains $r_{1}$ edges from the outer cycle, $r_{2}$ edges from the inner cycle and $r_{3}$ spokes, and that $r_{1}+r_{2}+r_{3} \leq k$. Since the outer and the inner cycles are negative, $r_{1}$ and $r_{2}$ must be odd and so $r_{3} \leq k-2$.

Now the $r_{1}$ negative edges of the outer cycle can make $r_{1}$ quadrangles negative. The $r_{2}$ negative edges of the inner cycle can make $r_{2}$ quadrangles negative. Further, the $r_{3}$ negative spokes can make at most $2r_{3}$ quadrangles negative. Thus the number of negative quadrangles in $E_{-}^{3}$ is at most $r_{1}+r_{2}+2r_{3}$, and $r_{1}+r_{2}+2r_{3} = (r_{1}+r_{2}+r_{3})+ r_{3} \leq k+r_{3} \leq k+k-2 = 2k-2$. However, as $E_{-}^{3} \sim E_{-}^{2}$, all the quadrangles must remain negative in $E_{-}^{3}$. Therefore any signature equivalent to $E_{-}^{2}$ must have at least $k+1$ edges. Hence $l[P_{2k,1}, E_{-}^{2}] = k+1$. This completes the proof.
\end{proof} 

From Theorem~\ref{gcd(n,k)=1} and Lemma~\ref{tight1}, we find the exact value of $D(P_{n,1})$.

\begin{theorem}\label{D(P_{n,1})}
Let $n\geq3$. Then $D(P_{n,1})=\left\lfloor \frac{n}{2} \right\rfloor+1.$
\end{theorem}

In the following lemma, we show that there exists a signature $E_{-}$ such that $l[P_{2m+1,2}, E_{-}] = m+1$, where $m\geq2$. 

\begin{lemm}\label{tight2}
For $m \geq 2$, there exists is a signature $E_{-}$ such that $l[P_{2m+1,2}, E_{-}] = m+1$.
\end{lemm}
\begin{proof}
For $m=2$, the graph $P_{5,2}$ is the Petersen graph. In \cite{T.Zaslavsky}, the author proved that there is a signed Petersen graph of frustration index three. Therefore the result is true for $m=2$.

Now let $m\geq3$ be odd. Consider the signed generalised Petersen graph $[P_{2m+1,2}, E_{-}^{1}]$, where \linebreak[4] $E_{-}^{1} = \{u_{0}v_{0}, u_{1}v_{1}, u_{4}v_{4}, u_{5}v_{5},...,u_{2m-6}v_{2m-6},u_{2m-5}v_{2m-5}, u_{2m-2}v_{2m-2}, v_{2m-3}v_{2m-1}\}$. It is clear that $|E_{-}^{1}|= m+1$. Further, $P_{2m+1,2}$ have exactly $2m+1$ cycles of length five. Note that with one negative edge, at most two 5-cycles can be made negative and therefore with $m$ or less than $m$ edges, at most $2m$ 5-cycles can be made negative. But $E_{-}^{1}$ makes all the 5-cycles negative. Thus $E_{-}^{1}$ is minimum. Hence $l[P_{2m+1,2}, E_{-}^{1}] = m+1$.

Let $m \geq 4$ be even. Consider the signed generalised Petersen graph $[P_{2m+1,2}, E_{-}^{2}]$, where \linebreak[4] $E_{-}^{2} = \{u_{0}v_{0}, u_{1}v_{1}, u_{4}v_{4}, u_{5}v_{5},...,u_{2m-4}v_{2m-4},u_{2m-3}v_{2m-3},v_{2m-2}v_{2m}\}$. It is clear that $|E_{-}^{2}|= m+1$, and $E_{-}^{2}$ makes all the 5-cycles negative. By similar argument as above paragraph, it can be shown that $E_{-}^{2}$ is minimum. Hence $l[P_{2m+1,2}, E_{-}^{2}] = m+1$. This completes the proof.
\end{proof}

From Theorem~\ref{gcd(n,k)=1} and Lemma~\ref{tight2}, we find the exact value of $D(P_{2m+1,2})$, as given in Theorem~\ref{D(P_{2m+1,2})}.

\begin{theorem}\label{D(P_{2m+1,2})}
Let $m \geq 2$. Then $D(P_{2m+1,2}) = m+1.$
\end{theorem}

In the next lemma, we show that $l[P_{4m-1,3}, E_{-}] = 2m$, where $\gcd(4m-1,3)=1~\text{and}~m\geq2$.

\begin{lemm}\label{tight3}
Let $m \geq 2$ and $\gcd(4m-1,3)=1$. Then there exists is a signature $E_{-}$ such that $l[P_{4m-1,3}, E_{-}] = 2m$.
\end{lemm}
\begin{proof}
Consider the signed generalized Petersen graph $[P_{4m-1,3}, E_{-}]$, where the signature $E_{-}$ is \linebreak[4] $E_{-}= \{u_{0}v_{0},u_{2}v_{2},u_{4}v_{4},...,u_{4m-6}v_{4m-6},u_{4m-4}v_{4m-4},u_{4m-3}u_{4m-2}\}$. Clearly $|E_{-}| = 2m$. Note that $P_{4m-1,3}$ has exactly $4m-1$ cycles of length six. Further, with $2m-1$ or less than $2m-1$ negative edges at most $4m-2$ cycles of length six can be made negative. But $E_{-}$ makes all the 6-cycles negative. Therefore signature $E_{-}$ is minimum. This implies that $l[P_{4m-1,3}, E_{-}] = 2m$ and the proof is complete.
\end{proof}

In the following theorem, that follows from Theorem~\ref{gcd(n,k)=1} and Lemma~\ref{tight3}, we see that the exact value of $D(P_{4m-1,3})$ is $2m$, where $\gcd(4m-1,3)=1$. 

\begin{theorem}\label{D(P_{4m-1,3})}
Let $m\geq2$ be such that $\gcd(4m-1,3)=1$. Then $D(P_{4m-1,3}) = 2m.$
\end{theorem}

\section{Upper Bound of $D(P_{n,k})$ for $\gcd(n,k) \geq 2$}\label{Bound_gcd(n,k) = d}

In this section, we get an upper bound for the maximum frustration index of $P_{n,k}$, where $\gcd(n,k) = d \geq 2$. Since $\gcd(n,k) = d$, so $P_{n,k}$ has $d$ pairwise disjoint inner cycles of length $\frac{n}{d}$. Let us denote these $d$ cycles by $C_{I_1}, C_{I_2},..., C_{I_d}$, where $C_{I_i} = v_{i-1}v_{i+k-1}v_{i+2k-1}...v_{n-k+i-1}v_{i-1}$ for $1 \leq i \leq d$.

For $1 \leq i \leq d$, let $G_{I_i}$ be a subgraph of $P_{n,k}$ which is defined as follows. The vertex set of $G_{I_i}$ is given by $V(G_{I_i}) = \{v_{i-1},v_{i+k-1},v_{i+2k-1},...,v_{n-k+i-1}, u_{i-1},u_{i+k-1},u_{i+2k-1},...,u_{n-k+i-1} \}$ and the edge set of $G_{I_i}$ is given by the union of $E(C_{I_i})$ along with the set of all spokes incident to the vertices of $C_{I_i}$. We say $[G_{I_i}, E_{-}]$ is balanced or unbalanced according as $C_{I_i}$ is positive or negative, respectively. 

It is important to note that the graph $G_{I_i}$ is isomorphic to $G_{r}$ with $r = \frac{n}{d}$, where $G_{r}$ is given in Definition~\ref{$G_n$}. Let $[G_{I_i}, E_{-}]$ be a signed graph and we do not switch $u\text{-vertices}$ of $G_{I_i}$. Then by Lemma~\ref{lemma1o4}, the maximum number of negative edges in a minimum signature is either $\left\lfloor \frac{n}{2d} \right\rfloor +1$ or $\left\lfloor \frac{n}{2d} \right\rfloor$ according as $[G_{I_i}, E_{-}]$ is unbalanced or balanced, respectively. Now we proceed to determine an upper bound for the maximum frustration of $P_{n,k}$ in terms of $n$ and $d$, where $\gcd(n,k) = d\geq2$.

\begin{theorem}\label{boundgcd(n,k)=d}
Let $\gcd(n,k) = d\geq 2$. Then $D(P_{n,k}) \leq d\left\lfloor \frac{n}{2d} \right\rfloor +d +1.$
\end{theorem}
\begin{proof}
Let $[P_{n,k}, E_{-}]$ be a signed generalised Petersen graph. Let $E_{-}'$ be a minimum signature, equivalent to $E_{-}$. We consider the following cases to show that $|E_{-}'| \leq d\left\lfloor \frac{n}{2d} \right\rfloor +d +1$.

\vspace*{0.3cm} 
\noindent
\textbf{Case 1.} For each $1 \leq i \leq d$, let $C_{I_i}$ be positive in $[P_{n,k}, E_{-}]$. Without loss of generality, we assume all the edges of $C_{I_i}$ to be positive for $1 \leq i \leq d$. It is clear that all the edges of $E_{-}$ lie inside the subgraph $G_{o}$ of $P_{n,k}$. To reduce the signature $E_{-}$ further, we do not need to switch any $v\text{-vertices}$ and therefore from Lemma~\ref{lemma1o4}, the maximum number of negative edges in a minimum signature $E_{-}'$, switching equivalent to $E_{-}$, is $\left\lfloor \frac{n}{2} \right\rfloor +1$. Thus $|E_{-}'| \leq \left\lfloor \frac{n}{2} \right\rfloor +1 < d\left\lfloor \frac{n}{2d} \right\rfloor +d +1$. 

\vspace*{0.3cm} 
\noindent
\textbf{Case 2.} Let the cycle $C_{o}$ be positive and let $r$ number of $C_{I_i}\text{'s}$ be negative in $[P_{n,k}, E_{-}]$, where $1 \leq r \leq d$. Without loss of generality, we assume all the edges of $C_{o}$ to be positive. So all the edges of $E_{-}$ lie in the union of $G_{I_i}\text{'s}$. To reduce $E_{-}$, we do not need to switch any $u\text{-vertices}$, so by Lemma~\ref{lemma1o4}, $r$ unbalanced $G_{I_i}\text{'s}$ can add at most $r(\left\lfloor \frac{n}{2d} \right\rfloor +1)$ edges in $E_{-}'$ and $d-r$ balanced $G_{I_i}\text{'s}$ can add at most $(d-r)\left\lfloor \frac{n}{2d} \right\rfloor$ edges in $E_{-}'$. Thus we have $|E_{-}'| \leq r(\left\lfloor \frac{n}{2d} \right\rfloor +1) + (d-r)\left\lfloor \frac{n}{2d} \right\rfloor \leq d\left\lfloor \frac{n}{2d} \right\rfloor +r < d\left\lfloor \frac{n}{2d} \right\rfloor +d +1$, as $r \leq d$. Hence $|E_{-}'| < d\left\lfloor \frac{n}{2d} \right\rfloor +d +1$.

\vspace*{0.3cm} 
\noindent
\textbf{Case 3.} Let the cycle $C_{o}$ be negative and let $r$ number of $C_{I_i}\text{'s}$ be negative in $[P_{n,k}, E_{-}]$, where $1 \leq r \leq d$. Without loss of generality, we assume that exactly one edge of $C_{o}$ is negative in $E_{-}$, and let that negative edge be $u_{0}u_{1}$. Thus all the edges of $E_{-}$, except $u_{0}u_{1}$, lie in the union of $G_{I_i}\text{'s}$. So by Lemma~\ref{lemma1o4}, it is easy to see that $|E_{-}'| \leq r(\left\lfloor \frac{n}{2d} \right\rfloor +1) + (d-r)\left\lfloor \frac{n}{2d} \right\rfloor +1 \leq d\left\lfloor \frac{n}{2d} \right\rfloor +d +1$, as $r \leq d$.

From these three cases, we conclude that $D(P_{n,k}) \leq d\left\lfloor \frac{n}{2d} \right\rfloor +d +1.$ This completes the proof.
\end{proof}

We note that if the inner cycles of a generalised Petersen graph are triangles then the bound obtained in Theorem~\ref{boundgcd(n,k)=d} can be improved. To see this fact, let $P_{3k,k}$ be a generalised Petersen graph, where $k \geq 2$. It is obvious that the inner cycles of $P_{3k,k}$ are triangles and there are $k$ such pairwise disjoint triangles induced by $v\text{-vertices}$. The subgraph $G_{I_1}$ of $P_{3k,k}$ is shown in Figure~\ref{Figure3}.

\begin{figure}[h]
\centering
\begin{tikzpicture}[scale=0.6]
\node[vertex] (v1) at (0,0) {};
\node at (-0.3,0.4) {$v_{2k}$};
\node[vertex] (v2) at (3,0) {};
\node at (3.25,0.25) {$v_{k}$};
\node[vertex] (v3) at (1.5,2.5) {};
\node at (2,2.5) {$v_{0}$};
\node[vertex] (v4) at (1.5,4) {};
\node at (2,4) {$u_{0}$};
\node[vertex] (v5) at (4.6,-1) {};
\node at (4.8,-0.65) {$u_{k}$};
\node[vertex] (v6) at (-1.6,-1) {};
\node at (-1.8,-0.6) {$u_{2k}$};

\foreach \from/\to in {v1/v3,v1/v2,v2/v3,v3/v4,v2/v5,v1/v6} \draw (\from) -- (\to);

\end{tikzpicture}
\caption{The graph $G_{I_1}$} \label{Figure3}
\end{figure}
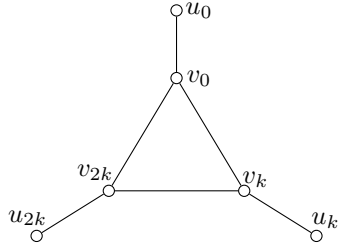

For $1 \leq i \leq k$, let $[G_{I_i}, E_{-}]$ be a signed graph and assume that switching by $u\text{-vertices}$ is not allowed. Then by Lemma~\ref{lemma1o4}, the maximum number of negative edges in a minimum signature, switching equivalent to $E_{-}$, is either two or one according as $[G_{I_i}, E_{-}]$ is unbalanced or balanced, respectively. Further, if $[G_{I_i}, E_{-}]$ is unbalanced  and a minimum equivalent signature $E_{-}'$ has two edges then the edges of $E_{-}'$ can be chosen to be $v_{i-1}v_{i+k-1}$ and $u_{i+2k-1}v_{i+2k-1}$. We call the edge $v_{i-1}v_{i+k-1}$ a \emph{primary edge} and the edge $u_{i+2k-1}v_{i+2k-1}$ a \emph{secondary edge} of $G_{I_i}$. Therefore the secondary edge of any minimum signature of unbalanced $[G_{I_i}, E_{-}]$ must belongs to the set $S$, where $S=\{u_{2k}v_{2k}, u_{2k+1}v_{2k+1},...u_{3k-1}v_{3k-1}\}$.

Let $[P_{3k,k}, E_{-}]$ be a signed graph and $E_{-}'$ be a minimum signature equivalent to $E_{-}$. We have the following observations.

\vspace*{0.2cm}
\noindent
\textbf{Observation 1.} For some $1 \leq i \leq k$, if $C_{I_i}$ is positive in $[P_{3k,k}, E_{-}]$ then the corresponding $G_{I_i}$ contributes at most one edge in $E_{-}'$. This holds true because of Lemma~\ref{lemma1o4}.

\vspace*{0.2cm}
\noindent
\textbf{Observation 2.} For $1 \leq r \leq k$, let $r$ number of $G_{I_i}\text{'s}$ be unbalanced in $[P_{3k,k}, E_{-}]$, each of which contributes primary as well as secondary edge in $ E_{-}'$. Recall that the secondary edge of an unbalanced $G_{I_i}$ must belongs to the set $S$, where $S=\{u_{2k}v_{2k}, u_{2k+1}v_{2k+1},...,u_{3k-1}v_{3k-1}\}$. If $r \leq \left\lfloor \frac{k}{2} \right\rfloor +1$ then such $G_{I_i}\text{'s}$ contribute at most $2r$ negative edges to $E_{-}'$. If $r \geq \left\lfloor \frac{k}{2} \right\rfloor +2$ then by switching each vertex of the set $\{u_{2k},u_{2k+1},...,u_{3k-1}\}$ we will have at most $\left\lfloor \frac{k}{2} \right\rfloor +1$ negative edges instead of having $r$ secondary edges. This ultimately implies that $r$ unbalanced $G_{I_i}\text{'s}$ can contribute at most $r+\left\lfloor \frac{k}{2} \right\rfloor +1$ negative edges to a minimum signature.

\begin{theorem}\label{boundP_{3k,k}}
For $k \geq 2$, $D(P_{3k,k}) \leq \left\lfloor \frac{3k}{2} \right\rfloor +2$.
\end{theorem}
\begin{proof}
Let $[P_{3k,k}, \Sigma]$ be a signed graph and $E_{-}'$ be a minimum signature, equivalent to $E_{-}$. To prove the theorem it is enough to show that $|E_{-}'| \leq \left\lfloor \frac{3k}{2} \right\rfloor +2$. We consider the following cases.

\vspace*{0.3cm}
\noindent
\textbf{Case 1.} For each $1 \leq i \leq k$, let $C_{I_i}$ be positive in $[P_{3k,k}, E_{-}]$. Without loss of generality, assume that all the edges of inner triangles are positive. Therefore all the edges of $E_{-}$ lie inside the graph $G_{o}$. By Lemma~\ref{lemma1o4}, the maximum number of negative edges in a minimum signature $E_{-}'$ is $\left\lfloor \frac{3k}{2} \right\rfloor +1$. Hence $|E_{-}'| \leq \left\lfloor \frac{3k}{2} \right\rfloor +1$.

\vspace*{0.3cm}
\noindent
\textbf{Case 2.} Let $C_{o}$ be positive and let $r$ number of $C_{i}\text{'s}$ be negative in $[P_{3k,k}, E_{-}]$, where $1 \leq r \leq k$. Without loss of generality, assume that all the edges of $C_{o}$ are positive. If some vertex $v_{i}$ is incident to more than one negative edge, then we switch that vertex $v_{i}$ and we keep doing this operation until we get a minimum signature $E_{-}'$. 

Since switching does not change the signs of cycles, we have $r$ unbalanced $G_{I_i}\text{'s}$ and $k-r$ balanced $G_{I_i}\text{'s}$ in $[P_{3k,k}, E_{-}']$. Therefore the maximum number of negative edges in $[P_{3k,k}, E_{-}']$ is $(r+ \left\lfloor \frac{k}{2} \right\rfloor +1) + (k-r)$, where first summand occurs due to \textbf{observation 2} and second summand occurs due to \textbf{observation 1}. This implies that $|E_{-}'| \leq k+  \left\lfloor \frac{k}{2} \right\rfloor +1 = \left\lfloor \frac{3k}{2} \right\rfloor +1$.

\vspace*{0.3cm}
\noindent
\textbf{Case 3.} Let $C_{o}$ be negative and let $r$ number of $C_{i}\text{'s}$ be negative in $[P_{3k,k}, E_{-}]$, where $1 \leq r \leq k$. Without loss of generality, assume that exactly one of the edges of $C_{o}$ is negative. By similar argument as in Case 2, it can be easily shown that $|E_{-}'| \leq \left\lfloor \frac{3k}{2} \right\rfloor +2$.

In all possible cases, we observe that the maximum number of edges in a minimum signature, equivalent to $E_{-}$, is $\left\lfloor \frac{3k}{2} \right\rfloor +2$. Since $E_{-}$ is arbitrary, we have $D(P_{3k,k}) \leq \left\lfloor \frac{3k}{2} \right\rfloor +2.$ This completes the proof.
\end{proof}

\section{Conclusion}

Among the few questions raised by this research, the followings are of particular interest to the authors.

1. For $k=1,2,3$, we obtained the exact value of $D(P_{n,k})$ where $\gcd(n,k)=1$. What is the exact value of $D(P_{n,k})$ in other cases where $\gcd(n,k)=1$?

2. In Theorem~\ref{boundgcd(n,k)=d}, we obtained $D(P_{n,k}) \leq d\left\lfloor \frac{n}{2d} \right\rfloor +d +1$. One can try to improve this bound to $\left\lfloor \frac{n}{2} \right\rfloor +1$.

From Theorem~\ref{bound0}, we have $l[G, E_{-}] \leq \frac{n}{2}$, where $[G, E_{-}]$ is any signed cubic graph. Can this bound be improved? In~\cite{Sivaraman}, the author give a family of signed cubic graphs to show that without any additional restrictions on the underlying graph, the bound cannot be improved. For example, consider the disjoint union of $k$ copies of $K_4$ with two disjoint negative edges in each $K_4$. But if we restrict to connected graphs then the following natural questions can be asked: 

3. Which cubic connected signed graphs attain the bound $\frac{n}{2}$?

4. Which cubic signed bipartite graphs attain the bound $\frac{n}{2}$?

5. Which cubic planar connected signed graphs attain the bound $\frac{n}{2}$?

\vspace*{0.3cm}
\noindent
\textbf{Acknowledgment.} We wholeheartedly thank Professor Thomas Zaslavsky for suggesting possible research directions and several comments to improve the presentation of this manuscript.

\end{document}